\documentclass[a4paper]{amsart}
\usepackage{graphicx}
\usepackage{amssymb,amscd} 
\usepackage{stmaryrd}
\usepackage{dsfont}
\usepackage[mathcal]{euscript}
\usepackage{tikz}
\usepackage{tikz-cd}
\tikzset{labl/.style={anchor=south, rotate=270, inner sep=.5mm}}
\usepackage{hyperref}
\hypersetup{%
  bookmarksnumbered=true,%
  colorlinks=true,%
  linkcolor=blue,%
  citecolor=blue,%
  filecolor=blue,%
  menucolor=blue,%
  urlcolor=blue,%
  bookmarksopen=true,%
  bookmarksdepth=2,%
  pageanchor=true}

\title{Completions of triangulated categories}

\author{Henning Krause}
\address{Fakult\"at f\"ur Mathematik\\
Universit\"at Bielefeld\\ D-33501 Bielefeld\\ Germany}
\email{hkrause@math.uni-bielefeld.de}

\numberwithin{equation}{section}

\theoremstyle{plain}

\newtheorem{prop}[equation]{Proposition}
\newtheorem{lem}[equation]{Lemma} 

\newtheorem{cor}[equation]{Corollary}

\theoremstyle{definition}
\newtheorem{defn}[equation]{Definition}
\newtheorem{exm}[equation]{Example}
\newtheorem{probl}[equation]{Problem}

\theoremstyle{remark}
\newtheorem{rem}[equation]{Remark}

\newcommand{\art}{\operatorname{art}}

\newcommand{\Cau}{\operatorname{Cau}}

\newcommand{\Coh}{\operatorname{Coh}}

\newcommand{\colim}{\operatorname*{colim}}
\newcommand{\Cone}{\operatorname{Cone}}

\newcommand{\End}{\operatorname{End}}
\newcommand{\END}{\operatorname{\mathcal{E}\!\!\;\mathit{nd}}}
\newcommand{\Ext}{\operatorname{Ext}}

\newcommand{\fl}{\operatorname{fl}}

\newcommand{\Fun}{\operatorname{Fun}}

\newcommand{\hocolim}{\operatorname*{hocolim}}

\newcommand{\Hom}{\operatorname{Hom}}
\newcommand{\HOM}{\operatorname{\mathcal{H}\!\!\;\mathit{om}}}
\newcommand{\id}{\operatorname{id}}

\newcommand{\Ind}{\operatorname{Ind}}
\newcommand{\inj}{\operatorname{inj}}
\newcommand{\Inj}{\operatorname{Inj}}

\newcommand{\Lex}{\operatorname{Lex}}

\renewcommand{\mod}{\operatorname{mod}}
\newcommand{\Mod}{\operatorname{Mod}}

\newcommand{\noeth}{\operatorname{noeth}}

\newcommand{\Ph}{\operatorname{Ph}}
 
\newcommand{\proj}{\operatorname{proj}} 
\newcommand{\Proj}{\operatorname{Proj}}

\newcommand{\soc}{\operatorname{soc}}

\newcommand{\St}{\operatorname{St}}
\newcommand{\StMod}{\operatorname{StMod}}
\newcommand{\stmod}{\operatorname{stmod}}

\newcommand{\thick}{\operatorname{thick}}

\newcommand{\Ab}{\mathrm{Ab}}

\newcommand{\cau}{\mathrm{Cau}}

\newcommand{\op}{\mathrm{op}}

\newcommand{\Set}{\mathrm{Set}}

\newcommand{\iso}{\xrightarrow{\raisebox{-.4ex}[0ex][0ex]{$\scriptstyle{\sim}$}}}

\newcommand{\longiso}{\xrightarrow{\ \raisebox{-.4ex}[0ex][0ex]{$\scriptstyle{\sim}$}\ }}

\newcommand{\lto}{\longrightarrow}

\newcommand{\xto}{\xrightarrow}
\newcommand*{\intref}[2]{\def\tmp{#1}\ifx\tmp\empty\hyperref[#2]{\ref*{#2}}\else\hyperref[#2]{#1~\ref*{#2}}\fi}

\def\A{\mathcal A} 
 
\def\C{\mathcal C}
\def\D{\mathcal D}

\def\I{\mathcal I}
\def\J{\mathcal J}

\def\P{\mathcal P}
 
\def\T{\mathcal T}

\def\X{\mathcal X}

\def\bfi{\mathbf i}

\def\bfD{\mathbf D} 
\def\bfK{\mathbf K}

\def\bbN{\mathbb N}

\def\bbZ{\mathbb Z}

\newcommand{\fm}{\mathfrak{m}}

\def\a{\alpha}
\def\b{\beta}

\def\g{\gamma}
\def\p{\phi}
\def\s{\sigma}

\def\Ga{\Gamma}
\def\La{\Lambda}
\def\Si{\Sigma}

\begin{document}

\begin{abstract}
These notes for a master class at Aarhus University (March 22--24,
2023) provide an introduction to the theory of completion for
triangulated categories.
\end{abstract}

\date{\today}

\maketitle
\tableofcontents

\setcounter{section}{-1}

\section{Introduction}

Completions arise in all parts of mathematics. For example, the real
numbers are the completion of the rationals. This construction uses
equivalence classes of Cauchy sequences and is due to Cantor
\cite{Ca1872} and Méray \cite{Me1869}.  There is an obvious
generalisation in the setting of metric spaces, leading to the
completion of a metric space. Completions of rings and modules play an
important role in algebra and geometry. For categories there is the
notion of ind-completion due to Grothendieck and Verdier \cite{GV1972}
which provides an embedding into categories with filtered colimits.

In these notes we combine all these ideas to study completions of
triangulated categories. We do not offer an elaborated theory and
rather look at various examples which arise naturally in
representation theory of algebras.

These notes are based on three lectures; they are divided into nine
sections (roughly three per lecture). The first three sections provide
some preparations. In particular, we introduce the ind-completion of a
category and explain completions in the context of modules over
commutative rings.  In \S\ref{se:completion} we propose a definition
of a partial completion for triangulated categories; roughly speaking
these are full subcategories of the ind-completion which admit a
triangulated structure. In \S\ref{se:pure-inj} we provide various
criteria for an exact functor between triangulated categories to be a
partial completion, and \S\ref{se:t-vs-c} discusses a general set-up
in terms of compactly generated triangulated categories which covers
many of the examples that arise in nature. The final sections provide
detailed proofs for some of the examples and discuss the role of
enhancements.

It is my pleasure to thank Charley Cummings, Sira Gratz, and Davide
Morigi for organising the \emph{Categories, clusters, and completions
  master class} at Aarhus University. In particular, I am very
grateful to them for suggesting the topic of this series of lectures.
My own interest in this subject arose from my collaboration with Dave
Benson, Srikanth Iyengar, and Julia Pevtsova; I am most grateful for
their inspiration and for their helpful comments on these notes. Last
but not least let me thank Amnon Neeman for his comments, and let me
recommend the notes from his lectures for another perspective on this
fascinating subject.

This work was partly supported by the Deutsche Forschungsgemeinschaft
(SFB-TRR 358/1 2023 - 491392403).

\section{The ind-completion of a category}

Let $\C$ be an essentially small category. We write
$\Fun(\C^\op,\Set)$ for the category of functors $\C^\op\to\Set$.
Morphisms in $\Fun(\C^\op,\Set)$ are natural transformations. Note
that this category is complete and cocomplete, that is, all limits and
colimits exist and are computed pointwise in $\Set$.

For functors $E$ and $F$ we write $\Hom(E,F)$ for the set of morphisms
from $E$ to $F$.  The \emph{Yoneda functor}
\[\C\lto \Fun(\C^\op,\Set),\quad X\mapsto
  H_X:=\Hom(-,X)\]
is fully faithful; this follows from Yoneda's lemma which provides a
natural bijection
\[\Hom(H_X,F)\iso F(X).\]

Any  functor $F\colon\C^\op\to\Set$ can be written
canonically as a colimit of representable functors
\begin{equation}\label{eq:slice}
\colim_{H_X\to F} H_X \iso F
\end{equation}
where the colimit is taken over the \emph{slice category} $\C/F$; see
\cite[Proposition~3.4]{GV1972}. Objects in $\C/F$ are morphisms
$ H_X\to F$ where $X$ runs through the objects of $\T$. A morphism in
$\C/F$ from $H_X\xto{\phi} F$ to $H_{X'}\xto{\phi'} F$ is a morphism
$\alpha \colon X\to X'$ in $\C$ such that $\phi'
H_{\alpha}=\phi$.

\begin{defn}
  A \emph{filtered colimit} in a category $\D$ is the colimit of a
functor $\I\to \D$ such that the category $\I$ is
\emph{filtered}, that is
\begin{enumerate}
\item the category is non-empty,
\item  given objects $i,i'$ there is an object $j$ with
  morphisms $i\to j\leftarrow i'$, and
\item  given morphisms
$\a,\a'\colon i\to j$ there is a morphism $\b\colon j\to k$ such that
$\b\a=\b\a'$.
\end{enumerate}
\end{defn}

\begin{rem}
  A partially ordered set $(I,\le)$ can be viewed as a category: the
objects are the elements of $I$ and there is a unique morphism
$i\to j$ whenever $i\le j$. This category is filtered if and only if
$(I,\le)$ is non-empty and \emph{directed}, that is, for each pair of
elements $i,i'$ there is an element $j$ such that $i,i'\le j$. When
the colimit of a functor $\I\to\D$ is given by a directed partially
ordered set, this colimit is also called \emph{directed colimit} (or
confusingly \emph{direct limit}).

For each essentially small filtered category $\I$ there exists a
functor $\p\colon \J\to\I$ such that $\J$ is the category
corresponding to a directed partially ordered set and any functor
$X\colon \I\to \D$ induces an isomorphism
\[\colim_{j\in\J} X(\p(j))\iso\colim_{i\in\I}X(i).\] This fact will not be needed, but it may be useful to
know; see \cite[Proposition~8.1.6]{GV1972}.
\end{rem}

Let us get back to functors $F\colon \C^\op\to\Set$. It is not
difficult to show that $F$ is a filtered colimit of representable
functors if and only if the slice category $\C/F$ is filtered.

\begin{defn}
  The \emph{ind-completion} of $\C$ is the category of functors
  $F\colon \C^\op\to\Set$ that are filtered colimits of representable
  functors.  We denote this category by $\Ind\C$; it is a
category with filtered colimits and the Yoneda functor $\C\to\Ind\C$ is
the universal functor from $\C$ to a category with filtered colimits.
\end{defn}

It is convenient to identify $\C$ with the full subcategory of
representable functors in $\Ind\C$.  Let $X=\colim_i X_i$ and
$Y=\colim_j Y_j$ be objects in $\Ind\C$, written as filtered colimits
of objects in $\C$.

\begin{lem}\label{le:Hom-completion}
  We have  natural bijections
\[ \colim_i\Hom(C,X_i)\iso\Hom(C,\colim_i X_i)\quad \text{for each}
  \quad C\in\C\]
and  
   \[ \Hom(X,Y)\iso\lim_i\colim_j\Hom(X_i,Y_j).\]
\end{lem}

\begin{proof}
The first bijection is an immediate consequence of Yoneda's lemma. For
the second bijection we compute
\begin{align*}
  \Hom(X,Y)&=\Hom(\colim_i X_i, \colim_j\Hom Y_j)\\
           &\cong \lim_i\Hom(X_i, \colim_j Y_j)\\
           &\cong \lim_i\colim_j\Hom(X_i,Y_j).\qedhere
\end{align*}  
\end{proof}

The ind-completion takes a more familar form when we consider additive
categories. The following examples use the fact that a module $M$ over
any ring is finitely presented if and only if the representable
functor $\Hom(M,-)$ preserves filtered colimits.

\begin{exm}
  Let $A$ be a ring and $\C=\proj A$ the category of finitely
  generated projective $A$-modules. A theorem of Lazard says that a
  module is flat if and only if it is a filtered colimit of finitely
  generated free modules. Thus $\Ind\C$ identifies with the category
  of flat $A$-modules.
\end{exm}

\begin{exm}
  Let $A$ be a ring and $\C=\mod A$ the category of finitely
  presented $A$-modules. It is well known that any
  module is is a filtered colimit of finitely
  presented modules. Thus $\Ind\C$ identifies with the category
  of all $A$-modules.
\end{exm}

Here is a useful fact which connects the above examples; it is due to
Lenzing \cite[Proposition~2.1]{Le1983}. Let $\D\subseteq\C$ be a full
subcategory. The inclusion induces a fully faithful functor
$\Ind\D\to\Ind\C$; this follows easily from
Lemma~\ref{le:Hom-completion}. Thus we may view $\Ind\D$ as a full
subcategory of $\Ind\C$.

\begin{lem}\label{le:Lenzing}
  An object $X\in\Ind\C$ belongs to $\Ind\D$ if and only if each
  morphism $C\to X$ with $C\in\C$ factors through an object in $\D$.
\end{lem}
\begin{proof}
  If $X$ belongs to $\Ind\D$, then it follows from Lemma~\ref{le:Hom-completion} that
  each morphism from $C\in\C$ factors through an object in $\D$.  Now
  write $X=\colim_{C\to X} C$ as a filtered colimit of objects in $\C$ as in
  \eqref{eq:slice}, using the slice category $\C/X$ as index
  category. The condition that each morphism $C\to X$ with $C\in\C$
  factors through an object in $\D$ is precisely the fact that $\D/X$
  is a \emph{cofinal subcategory} of $\C/X$. Thus the induced morphism
  \[\colim_{D\to X} D\lto \colim_{C\to X} C=X\]
is an isomorphism; see \cite[Proposition~8.1.3]{GV1972}.
\end{proof}

The above criterion applied to the inclusion $\proj A\subseteq\mod A$
shows that a module is flat if and only if each morphism from a finitely
presented module factors through a finitely generated projective module.

\begin{exm}
  Let $\C$ be an essentially small additive category with cokernels. A
  functor $F\colon\C^\op\to\Ab$ is \emph{left exact} if it is additive
  and sends each cokernel sequence $X\to Y\to Z\to 0$ to an exact
  sequence $0\to FZ\to FY\to FX$ of abelian groups. One can show that
  $F$ is left exact if and only if it is a filtered colimit of
  representable functors; see \cite[Lemma~11.1.14]{Kr2022}. Thus $\Ind\C$
  identifies with the category $\Lex(\C^\op,\Ab)$ of left exact
  functors $\C^\op\to\Ab$. If $\C$ is abelian, then $\Lex(\C^\op,\Ab)$
  is an abelian Grothendieck category and the Yoneda functor is exact.
\end{exm}

\begin{exm}
  Let $A$ be a commutative noetherian ring. Let $\C=\fl A$ denote the
  category of finite length modules (i.e.\ the finitely generated
  torsion modules). Then $\Ind\C$ identifies with the category of all
  torsion $A$-modules. This follows for example by applying the
  criterion in Lemma~\ref{le:Lenzing} to the inclusion
  $\fl A\subseteq\mod A$.
\end{exm}

The next example is relevant because we wish to
study filtered colimits arising from (or even in) triangulated categories.

\begin{exm}\label{ex:coh}
  Let $\C$ be an essentially small triangulated category.  Then
  $\Ind\C$ identifies with the category $\Coh\C$ of cohomological
  functors $\C^\op\to\Ab$. Recall that $\C^\op\to\Ab$ is
  \emph{cohomological} if it is additive and sends exact triangles to
  exact sequences.
\end{exm}

\begin{proof}
  Any representable functor is cohomological, and taking filtered
  colimits (in the category $\Ab$) is exact. Thus a filtered colimit
  of representable functors is cohomological. Conversely, suppose that
  $F\colon \C^\op\to\Ab$ is cohomological. Then it easily checked that
  the slice category $\C/F$ is filtered.
\end{proof}

\section{The sequential completion of a category}

The ind-completion of a category allows one to take arbitrary filtered
colimits, so colimits of functors that are indexed by any filtered
category. In the following we restrict to colimits of functors (or
sequences) that are indexed by the natural numbers.

Let $\bbN=\{0,1,2,\ldots\}$ denote the set of natural numbers, viewed
as a category with a single morphism $i\to j$ if $i\le j$.

Now fix a category $\C$ and consider the category $\Fun(\bbN,\C)$ of
functors $\bbN\to\C$.  An object $X$ is nothing but a sequence of
morphisms $X_0\to X_1\to X_2\to \cdots$ in $\C$, and the morphisms
between functors are by definition the natural transformations.  We
call $X$ a \emph{Cauchy sequence} if for all $C\in\C$ the induced map
$\Hom(C,X_i)\to\Hom(C,X_{i+1})$ is invertible for $i\gg 0$. This
means:
\[\forall\; C\in\C\;\; \exists\; n_C\in\bbN\;\; \forall\; j\ge i\ge
  n_C\;\; \Hom(C,X_i)\xto{_\sim}\Hom(C,X_j).\]

Let $\Cau(\bbN,\C)$ denote the full subcategory consisting of all
Cauchy sequences.  A morphism $X\to Y$ is \emph{eventually invertible}
if for all $C\in\C$ the induced map $\Hom(C,X_i)\to\Hom(C,Y_i)$ is
invertible for $i\gg 0$. This means:
\[\forall\; C\in\C\;\; \exists\; n_C\in\bbN\;\; \forall\;  i\ge
  n_C\;\; \Hom(C,X_i)\xto{_\sim}\Hom(C,Y_i).\] Let $S$ denote the class
of eventually invertible morphisms in $\Cau(\bbN,\C)$.

\begin{defn}
  The \emph{sequential Cauchy completion} of $\C$ is the category
  \[\Ind_\cau\C:= \Cau(\bbN,\C)[S^{-1}]\] that is obtained from the
  Cauchy sequences by formally inverting all eventually invertible
  morphisms, together with the \emph{canonical functor}
  $\C\to\Ind_\cau\C$ that sends an object $X$ in $\C$ to the constant
  sequence $X\xto{\id} X\xto{\id} \cdots$.
\end{defn}

A  sequence $X\colon\bbN\to\C$ induces a functor
\[\widetilde X\colon\C^\op\lto\Set,\quad C\mapsto\colim_i\Hom(C,X_i),\]
and this yields a functor
\begin{equation*}\label{eq:seq-compl}
  \Ind_\cau\C\lto\Ind\C\subseteq\Fun(\C^\op,\Set),\quad
  X\mapsto\widetilde X,
\end{equation*}
because the assignment $X\mapsto\widetilde X$ maps eventually invertible
morphisms to isomorphisms.

\begin{prop}[{\cite[Proposition~2.4]{Kr2020}}]\label{pr:completion}
  The canonical functor $\Ind_\cau\C\to\Ind\C$ is fully
  faithful; it identifies $\Ind_\cau\C$ with the colimits of
  sequences of representable functors that correspond to Cauchy
  sequences in $\C$.\qed
\end{prop}

It turns out that the class of Cauchy sequences is too restrictive; we need
a more general notion of completion.

\begin{defn}
  The \emph{sequential completion} of $\C$ has as objects all
  sequences in $\C$, i.e.\ functors $\bbN\to\C$, and for objects $X,Y$
  set
   \[\Hom(X,Y)=\lim_i\colim_j\Hom(X_i,Y_j).\]
   This category is denoted by $\Ind_\bbN\C$, and for any class $\X$
   of objects the corresponding full subcategory is denoted by
   $\Ind_\X\C$.
 \end{defn}
 It follows from Lemma~\ref{le:Hom-completion} that the assignment
$X\mapsto\colim_i\Hom(-,X_i)$ induces a fully faithful functor
$\Ind_\bbN\C\to\Ind\C$. Thus we obtain canonical inclusions
\[\Ind_\X\C\subseteq \Ind_\bbN\C\subseteq \Ind\C.\]

\begin{exm}\label{ex:countable-envelope}
 Let $\C$ be an exact
 category and let $\X$ denote the class of sequences $X$ such that each
$X_i\to X_{i+1}$ is an admissible monomorphism. Then
$\C\tilde{\phantom{e}}:=\Ind_\X\C$ admits a canonical exact
structure and is called \emph{countable envelope} of $\C$
\cite[Appendix~B]{Ke1990}.
\end{exm}

\section{Completion of rings and modules}

Let $A$ be an associative ring. We consider the category $\Mod A$ of right
$A$-modules and the following full subcategories:
\begin{align*}
  \mod A &= \text{finitely presented $A$-modules}\\
  \proj A &= \text{finitely generated projective $A$-modules}\\
  \noeth A &= \text{noetherian $A$-modules (satisfying the ascending 
              chain
              condition)}\\
  \art A &= \text{artinian $A$-modules (satisfying the descending 
            chain
            condition)}\\
  \fl A&=\art A \cap \noeth A= \text{finite length $A$-modules} 
\end{align*}

\begin{defn}
  For an ideal $I\subseteq A$ the \emph{$I$-adic completion} of $A$ is
  the limit \[\widehat A:=\lim_{n\ge 0}A/I^n.\] Similarly for an
  $A$-module $M$ one sets
\[\widehat M:=\lim_{n\ge 0} M/M I^n.\]
\end{defn}

Let us consider the special case that $A$ is a commutative
noetherian local ring and $I=\fm$ its unique maximal ideal. We denote by
$E=E(A/\fm)$ the injective envelope of the unique simple
$A$-module. Then $D=\Hom(-,E)$ yields the \emph{Matlis duality}
$\Mod A\to\Mod A$, satisfying
\[\Hom(M,DN)\cong\Hom(N,DM)
  \qquad (M,N\in\Mod A).\]
Note that $M\iso D^2M$ when $M$ has finite length; this is easily
checked by induction on the length of $M$. Thus $D$ induces an equivalence
\[(\fl A)^\op\longiso\fl A.\]
For $n\ge 0$ we write $E_n=\Hom(A/\fm^n,E)$ and note that
\[E=\bigcup_{n\ge 0}E_n.\] In fact, the module $E$ is artinian and
each submodule $E_n$ is of finite length. Thus
\[\Hom(E,E)\cong \Hom(\colim_{n\ge 0}E_n,E)\cong
  \lim_{n\ge 0}\Hom(E_n,E)\cong \lim_{n\ge 0}A/\fm^n=\widehat A.\]
In particular, each Matlis dual module $DM$ is canonically an
$\widehat A$-module via the map
$\widehat A\iso \End(E)$. Thus Matlis duality yields the
following commutative diagram.
\[\begin{tikzcd}
(\fl A)^\op\arrow["\sim" labl,"D" ']{d}\arrow[tail]{rr}&&(\art A)^\op \arrow["\sim" labl,"D" ']{d}\\
\fl A\arrow[tail]{r} &\noeth A\arrow{r}{\widehat{}}&\noeth\widehat A
\end{tikzcd}\]

Given a module $M$,  the
\emph{socle} $\soc M$ is the sum of all simple submodules. One
defines inductively $\soc^n M\subseteq M$ for $n\ge 0$ by setting $\soc^0 M=0$,
and $\soc^{n+1}M$ is given by the exact sequence
\[0\lto\soc^n M\lto\soc^{n+1} M\lto\soc (M/\soc^n M)\lto 0.\]

Recall that a ring $A$ is \emph{semi-local} if $A/J(A)$ is a
semisimple ring, where $J(A)$ denotes the Jacobson radical of $A$. In
that case we have $\soc M\cong\Hom(A/J(A),M)$ for every $A$-module
$M$.

\begin{prop}[{\cite[Proposition~3.4]{Kr2020}}]\label{pr:finite-length}
  Let $A$ be a commutative noetherian semi-local ring. Then the sequential
Cauchy  completion of $\fl A$ identifies with $\art A$.
\end{prop}

\begin{proof}
  Set $\C=\fl A$.  The assignment
  $X\mapsto \bar X:=\colim_n X_n$ yields a fully faithful functor
  \[\Ind_\cau\C\subseteq\Ind\C\lto\Ind(\mod A)=\Mod A.\]
  
  It is well known that an $A$-module $M$ is artinian if and only if
  $M$ is the union of finite length modules and $\soc M$ has finite
  length \cite[Proposition~2.4.20]{Kr2022}.  In that case the socle
  series $(\soc^nM)_{n\ge 0}$ of $M$ yields a Cauchy sequence in $\C$
  with $\colim_n(\soc^nM)=M$.

  Now let $X\in\Ind_\cau\C$. Then every finitely generated submodule
  of $\bar X$ has finite length, so $\bar X$ is a union of finite
  length modules. Also $\soc\bar X$ has finite length,
  since
  \[\soc\bar X\cong\Hom(A/J(A),\bar X)\cong\colim_n\Hom(A/J(A),
    X_n).\] Thus $\bar X$ is artinian.
\end{proof}

\begin{rem}\label{re:ind-adic}
  Completions of modules or categories of modules will serve as model
  for completions of triangulated categories. We have seen two types
  of completions: the \emph{ind-completion} (the inclusion
  $\fl A\to\art A$) and the \emph{adic completion} (the functor
  $\noeth A\to\noeth\widehat A$). Both types of completions have their
  analogues when we consider triangulated categories.
\end{rem}

\section{Completion of triangulated categories}\label{se:completion}

We are ready to propose a definition of `completion' for a triangulated
category. Roughly speaking it is a triangulated approximation of the
ind-completion. In fact, it will be rare that the ind-completion of a
triangulated category admits a triangulated structure, but it does
happen that certain full subcategories are triangulated.

Let $\C$ be an essentially small triangulated category. We denote by
$\Coh\C$ the category of cohomological functors $\C^\op\to\Ab$. In
Example~\ref{ex:coh} we have already seen that $\Coh\C$ equals
$\Ind\C$. An exact functor $f\colon\C\to\D$ between 
triangulated categories induces the \emph{restriction}
\[f_*\colon \D\lto\Coh\C,\quad X\mapsto\Hom(-,X)\circ f.
\] 

\begin{defn}
  We call a fully faithful exact functor $f\colon\C\to\D$ a
  \emph{partial completion} of the triangulated category $\C$ if the
  restriction $f_*\colon\D\to\Coh\C$ is fully faithful.  The
  completion is called \emph{sequential} if the above functor factors
  through $\Ind_\bbN\C$, and it is \emph{Cauchy sequential} if the
  functor factors through $\Ind_\cau\C$.
\end{defn}

A partial completion of $\C$ is far from unique. But depending on the context
there are often natural choices. An essential feature of a partial completion
is the fact that any object in $\D$ can be written canonically as a
filtered colimit of objects in the image of $f$. Suppose for
simplicity that $f$ is an inclusion. Then we have for each object
$X\in\D$ an isomorphism
\[\colim_{C\to X}C\iso X\]
where $C\to X$ runs through all morphisms in $\D$ with $C\in\C$. This
follows from \eqref{eq:slice}. Moreover, using
Lemma~\ref{le:Hom-completion} we can compute morphisms in $\D$ via
\[\Hom(X,X')\cong\lim_{C\to X}\colim_{C'\to X'}\Hom(C,C').\]
If the completion is sequential, then there is for each object
$X\in\D$ a sequence $C_0\to C_1\to C_2\to\cdots$ in $\C$ such
that \[\colim_n C_n\iso X.\] We note that these filtered colimits are
taken in $\D$; they exist because $\D$ identifies with a full
subcategory of the ind-completion of $\C$.

 Examples of partial completions arise from derived categories of exact
 subcategories.  For an exact category $\A$ we write $\bfD(\A)$ for
 its derived category and $\bfD^b(\A)$ for the full subcategory of
 bounded complexes.

\begin{exm}
  For a commutative noetherian ring $A$ the inclusion
  $\bfD^b(\fl A)\to \bfD^b(\art A)$ is a sequential partial completion.
\end{exm}

Recall that a ring $A$ is \emph{right coherent} if the
category $\mod A$ of finitely presented $A$-modules is abelian.

\begin{exm}\label{ex:coherent}
  For a right coherent ring $A$ the inclusion
  $\bfD^b(\proj A)\to\bfD^b(\mod A)$ is a Cauchy sequential partial
  completion.
\end{exm}

The assumption on the ring $A$ to be right coherent is not essential.
For an arbitrary ring one takes instead of $\mod A$ the exact category
of $A$-modules $M$ that admit a projective resolution
\[\cdots \lto P_1\lto P_0\lto M\lto 0\] such that each $P_i$ is
finitely generated; it is the largest full exact subcategory of
$\Mod A$ containing $\proj A$ and having enough projective objects.

We will return to these examples and provide full proofs. In fact, the proofs
require the study of compactly generated triangulated categories.

\section{Completion of compact objects and pure-injectivity}\label{se:pure-inj}

Partial completions of an essentially small triangulated category $\C$ often
arise as full triangulated subcategories of a compactly generated
triangulated category $\T$ such that $\C$ equals the subcategory of
compact objects.

Let $\T$ be a triangulated category that admits arbitrary
coproducts. An object $X$ in $\T$ is called \emph{compact} if the
functor $\Hom(X,-)$ preserves all coproducts. We denote by $\T^c$
the full subcategory of compact objects and note that it is a thick
subcategory of $\T$. The triangulated category $\T$ is \emph{compactly
  generated} if $\T^c$ is essentially small and if $\T$ has no proper
localising subcategory containing $\T^c$.

From now on fix a compactly generated triangulated category $\T$ and set
$\C=\T^c$.

\begin{defn}
 The  functor
\[\T\lto \Coh\C,\quad X\mapsto
  h_X:=\Hom(-,X)|_\C\]
is called \emph{restricted Yoneda functor}. The  induced
map \[\Hom(X,Y)\lto\Hom(h_X,h_Y)\qquad (X,Y\in\T)\]
is in general neither injective nor surjective; its kernel is the
subgroup of \emph{phantom morphisms}.
\end{defn}

The category $\Coh\C$ is an extension closed subcategory of the
abelian category $\Mod\C$ (i.e.\ the category of additive functors
$\C^\op\to\Ab$). Thus it is an exact category (in the sense of
Quillen) with enough projective and enough injective objects. In fact,
the projective objects are of the form $h_X$ with $X$ a direct summand
of a coproduct of compact objects in $\T$; this follows from Yoneda's
lemma.  An application of Brown's representability theorem shows that
also the injective objects are of the form $h_X$ for an object $X$ in
$\T$. This leads to the notion of a pure-injective object.

\begin{defn}
An exact triangle $X\to Y\to Z\to$ in $\T$ is called \emph{pure-exact}
if the induced sequence $0\to h_X\to h_Y\to h_Z\to 0$ is exact. The
triangle \emph{splits} if $X\to Y$ is a split monomorphism,
equivalently if $Y\to Z$ is a split epimorphism. An object $X$ in $\T$
is \emph{pure-injective} if each pure-exact triangle $X\to Y\to Z\to$
in $\T$ splits.
\end{defn}

We have the following characterisation of pure-injectivity.

\begin{prop}[{\cite[Theorem~1.8]{Kr2000}}]\label{pr:pure-inj}
For an object $X$ in $\T$ the following are equivalent.
\begin{enumerate}
    \item  The map $\Hom(X',X)\to\Hom(h_{X'},h_{X})$ is
      bijective for all $X'\in\T$.
    \item The object $h_X$ is injective in $\Coh\C$.
    \item The object $X$ is pure-injective in $\T$.
    \item For each set $I$ the summation morphism $\coprod_I X\to X$
      factors through the canonical morphism
      $\coprod_I X\to \prod_IX$.  \qed
\end{enumerate}   
\end{prop}

The following immediate consequence motivates our interest in
pure-injectives.

\begin{cor}\label{co:pure-inj}
  Let $\D\subseteq\T$ be a triangulated subcategory containing all
  compact objects and consisting of pure-injective objects. Then the
  inclusion $\C\to\D$ is a partial completion.\qed
\end{cor}

Pure-injectivity is a useful homological condition but in practice
hard to check. In particular, there is no obvious triangulated
structure on pure-injectives.  We provide a criterion for a strong
form of pure-injectivity which is an analogue of artinianess for
modules.

Let $X\in\T$ and $C\in\T^c$. A \emph{subgroup of finite definition} is
a subgroup of $\Hom(C,X)$ that equals the image of an induced map
$\Hom(D,X)\to\Hom(C,X)$ given by a morphism $C\to D$ in $\T^c$. Note
that any subgroup of finite definition of $\Hom(C,X)$ is an 
$\End(X)$-submodule.  

\begin{lem}
  The subgroups of finite definition of $\Hom(C,X)$ are closed under
  finite sums and intersections. Thus they form a lattice.
\end{lem}
\begin{proof}
  Let $U_i$ be the image of $\Hom(D_i,X)\to\Hom(C,X)$ given by a
  morphism $C\to D_i$ in $\T^c$ ($i=1,2$). Then $U_1+U_2$ equals the image of the
  map induced by $C\to D_1\oplus D_2$. Now complete this to an exact
  triangle $C\to D_1\oplus D_2\to E\to$. Then $U_1\cap U_2$ equals the
  image of the map induced by $C\to D_1\to E$.
\end{proof}

We say that an object $X\in\T$ satisfies \emph{dcc on subgroups of
  finite definition} if for each compact object $C$ any chain of
subgroups of finite definition
\[\cdots \subseteq U_2 \subseteq U_1 \subseteq U_0=\Hom(C,X)\]
stabilises.

The following result goes back to Crawley-Boevey \cite[3.5]{CB1994}, who
proved this for locally finitely presented additive categories; see
also \cite[Theorem~12.3.4]{Kr2022}. The proof is quite involved; the
basic idea is to translate the descending
chain condition  into a noetherianess condition for some
appropriate Grothendieck category (a localisation of $\Mod\C$
cogenerated by $h_X$).

\begin{prop}\label{pr:dcc}
  For an object $X\in\T$ the following are equivalent.
  \begin{enumerate}
    \item $X$ is $\Sigma$-pure-injective, i.e.\ any coproduct
      of copies of $X$ is pure-injective.
      \item $X$ satisfies dcc
        on subgroups of finite definition.
      \item Every product of copies of $X$ decomposes into a coproduct
        of indecomposable objects with local endomorphism rings.
      \end{enumerate}
\end{prop}
\begin{proof}
  The category $\Coh\C$ is locally finitely presented and has
  products; so Crawley-Boevey's theory of purity can be applied. In particular,
  $X$ is pure-injective in $\T$ if and only if $h_X$ is pure-injective
  in $\Coh\C$, by Theorem~1 in \cite[3.5]{CB1994} and
  Proposition~\ref{pr:pure-inj}. Also, $h_X$ satisfies dcc on
  subgroups of finite definition in $\Coh\C$ if and only if $X$
  satisfies dcc on subgroups of finite definition in $\T$, as
  $\Hom(C,X)\cong\Hom(h_C,h_X)$ for each compact object $C$. Now the
  assertion follows from Theorem~2 in \cite[3.5]{CB1994}.
\end{proof}

Let $R$ be a commutative ring and suppose that $\C$ is
\emph{$R$-linear}. This means that there is a ring homomorphism
$R\to Z(\C)$ into the \emph{centre} of $\C$ (the ring of natural
transformations $\id_\C\to\id_\C$). In particular, for each pair of objects
$X,Y$ the group of morphisms $\Hom(X,Y)$ is naturally an
$R$-module.

We view $\C$ as a full subcategory of $\Coh\C$ and call an object $X$
in $\Coh\C$ or $\T$ \emph{artinian} over $R$ if $\Hom(C,X)$ is an
artinian $R$-module for all $C\in\C$ (via the canonical homomorphism
$R\to\End(C)$). Let $\art_R\C$ denote the full subcategory of
$R$-artinian objects in $\Coh\C$, and $\art_R\T$ denotes the full subcategory of
$R$-artinian objects in $\T$. One may drop $R$ when $R=Z(\C)$.

\begin{cor}
  The restricted Yoneda functor induces an equivalence
  $\art_R\T\iso\art_R\C$. In particular, the category $\art_R\C$
  admits a canonical triangulated structure that is induced from that
  of $\T$.
\end{cor}
\begin{proof}
 The restricted Yoneda functor induces  for $X\in\T$ and $C\in\C$ a bijection
  \[\Hom(C,X)\iso \Hom(h_C,h_X).\]
  Thus $X$ is $R$-artinian if and only if $h_X$ is $R$-artinian.  Now
  suppose that $X\in\Coh\C$ is $R$-artinian.  One has for $X\in\Coh\C$
  and $C\in\C$ the analogous concept of a subgroup of finite
  definition of $\Hom(C,X)$, and then artinianess over $R$ implies dcc
  on subgroups of finite definition. Thus it follows from Theorem~2 in
  \cite[3.5]{CB1994} that $X$ is pure-injective. The exact structure
  on $\Coh\C$ agrees with the pure-exact structure. Thus $X$ is an
  injective object and therefore of the form $h_{\bar X}$ for a
  pure-injective object $\bar X$ in $\T$ by
  Proposition~\ref{pr:pure-inj}. Clearly, $\bar X$ is $R$-artinian.

  Any $R$-artinian object $X\in\T$ is pure-injective by
  Proposition~\ref{pr:dcc}. Thus Proposition~\ref{pr:pure-inj} yields
  a bijection
 \[\Hom(X',X)\iso \Hom(h_{X'},h_X)\] for any pair $X,X'$ of
 $R$-artinian objects in $\T$.

 The $R$-artinian objects in $\T$ form a thick subcategory. Thus
 transport of structure provides a triangulated structure for
 $\art_R\C$.
\end{proof}

\begin{cor}\label{co:artinian}
  Let $\D\subseteq\T$ be a triangulated subcategory containing all
  compact objects and consisting of $R$-artinian objects. Then the
  inclusion $\C\to\D$ is a partial completion.
\end{cor}
\begin{proof}
The assertion is an immediate consequence of
Corollary~\ref{co:pure-inj} since each object in $\D$ is
$\Sigma$-pure-injective thanks to Proposition~\ref{pr:dcc}.
\end{proof}

We conclude with a couple of remarks. The first one addresses possible
triangulated structures for the ind-completion of an essentially small
triangulated category.

\begin{rem}
  There is a notion of a \emph{locally finite} triangulated category;
  see \cite{Kr2012,XZ2005}.  One way of defining this is that all
  cohomological functors into abelian groups (covariant or
  contravariant) are coproducts of direct summands of representable
  functors. An equivalent condition is that every short exact sequence
  of cohomological functors does split. Examples are the stable module
  category $\stmod A$ when $A$ is a self-injective algebra of finite
  representation type, or the derived category $\bfD^b(\mod A)$ of a
  hereditary algebra of finite representation type. Then we have
  equivalences $\Ind(\stmod A)\simeq\StMod A$ and
  $\Ind(\bfD^b(\mod A))\simeq \bfD(\Mod A)$. In particular, the
  ind-completions carry a triangulated structure; they are
  compactly generated and each pure-exact triangle splits.

  Now suppose that $\C$ is an essentially small triangulated category
  that is not locally finite. For example, let $\C=\stmod A$ when $A$
  is a finite dimensional self-injective algebra of infinite
  representation type. Passing to $\C^\op$ if necessary this means
  that not all objects in $\Ind\C$ are projective. So we find an exact
  sequence $0\to X\to Y\to Z\to 0$ which does not split. On the other
  hand, if $\Ind\C$ admits a triangulated structure, then each
  kernel-cokernel pair needs to split. It follows that $\Ind\C$ does
  not admit a triangulated structure.
\end{rem}

\begin{rem}
The suspension $\Si\colon\T\iso\T$ induces a $\bbZ$-grading of $\T$.
Let $Z^*(\T)$ denote the \emph{graded centre}.
This is a graded commutative ring and any  ring homomorphism
$R\to Z^*(\T)$ from a  graded commutative ring $R$ induces an
$R$-linear structure on the graded morphisms
\[ \Hom^*(X,Y):= \bigoplus_{i\in\bbZ}\Hom(X,\Si^i Y) 
\quad \text{for}\quad
X,Y\in\T.\]
In particular, the notion of an $R$-artinian object extends to the
graded setting. All statements and their proofs remain valid in this generality.
\end{rem}

\section{Torsion versus completion}\label{se:t-vs-c}

For any compactly generated triangulated category and any choice of
compact objects generating a localising subcategory, there is an
adjoint pair of functors that resembles derived torsion and completion
functors for the derived category of a commutative ring. 

Let $\T$ be a compactly generated triangulated category with
suspension $\Sigma\colon\T\iso\T$ and set $\C=\T^c$. We choose a thick
subcategory $\C_0\subseteq\C$ and denote by $\T_0\subseteq\T$ the
localising subcategory which is generated by $\C_0$. Note that
$(\T_0)^c=\C_0$.  The inclusion $\T_0\to\T$ admits a right adjoint, by Brown's
representability theorem, which we denote by $q\colon\T\to\T_0$. This
functor preserves coproducts and then another application of Brown's
representability theorem yields a right adjoint $q_\rho$ which is
fully faithful. Thus the
left adjoint $q_\lambda$ and the right adjoint $q_\rho$  provide two embeddings of
$\T_0$ into $\T$, and our notation suggests a symmetry which does not
give preference to any of the inclusions.
\begin{equation*}
\begin{tikzcd}[column sep = huge]
    {\T}
    \arrow[twoheadrightarrow]{r}[description]{q} &\T_0
    \arrow[tail,swap,yshift=1.5ex]{l}{q_\lambda}\arrow[tail,yshift=-1.5ex]{l}{q_\rho}
\end{tikzcd}
\end{equation*}

\begin{defn}
  The choice of  $\C_0\subseteq\C$ yields exact functors
\[\Ga := q_\lambda\circ q \qquad\text{and}\qquad \La := q_\rho\circ
  q \] which form an adjoint pair.  For any object $X$ in $\T$ the
unit $X\to\La X$ is called \emph{completion} and the counit
$\Ga X\to X$ is called \emph{torsion} or \emph{local cohomology} of
$X$.
\end{defn}

Note that these functors are idempotent, since
$\id_{\T_0}\iso q\circ q_\lambda$ and $q\circ q_\rho\iso \id_{\T_0}$.
From the definitions it is clear that the adjoint pair $(\Ga,\La)$
induces mutually inverse equivalences
\begin{equation}\label{eq:DG}
  \begin{tikzcd}[column sep = huge]
    \La\T \arrow[yshift=.75ex]{r}{\Ga} &{\Ga\T}
    \arrow[yshift=-.75ex]{l}{\La}
  \end{tikzcd}
\end{equation}
where $\La\T=\{X\in\T\mid X\iso\La X\}$ and $\Ga\T=\{X\in\T\mid \Ga
X\iso X\}$. 

We are interested in the completions of the compact objects from $\T$,
and we may view these as objects of $\Ga\T$, because of the above
equivalence \eqref{eq:DG}. Thus it is equivalent to look at the local
cohomology of the compact objects from $\T$.
Set \[\widehat\C:=\thick(\Ga\C)\simeq\thick(\La\C)\subseteq\La\T.\]
Then we obtain the following chain of inclusions.
\[\C_0=(\Ga\T)^c\subseteq\widehat\C\subseteq\Ga\T\subseteq\T\]
The next diagram shows how these various subcategories of $\T$ are
related. The inclusion $\T_0\to\T$ admits the right adjoint $\Ga$, and
we may think of its restriction $\C\to\widehat\C$ as a mock right adjoint
of the inclusion $\C_0\to\C$.
\[\begin{tikzcd}
\C_0  \arrow[tail]{d}\arrow[tail]{r}&\widehat\C \arrow[tail]{r}&\Ga\T \arrow[tail, xshift=-.75ex]{d}\\
\C   \arrow[tail]{rr}  \arrow{ur}&&\T
\arrow[twoheadrightarrow,xshift=.75ex,swap,"\Ga"]{u}
\end{tikzcd}\]

\begin{probl}
  Find a description of $\widehat\C$ for given triangulated categories $\C_0\subseteq\C$.
\end{probl}

Let us get back to the distinction between ind-completion and adic
completion (cf.\ Remark~\ref{re:ind-adic}). This carries over to our triangulated
setting; it means that we can approach $\widehat\C$ from two
directions, using either the inclusion
$\C_0\to\widehat\C$ or the  functor $\C\to\widehat\C$.

  We illustrate this with an example which explains the terminology;
  it goes back to work of Dwyer, Greenlees, and May
  \cite{DG2002,GM1992}.  Let $A$ be a commutative ring. We
  set $\T=\bfD(\Mod A)$ and identify $\bfD^b(\proj A)=\T^c=\C$ (the
  category of perfect complexes). Recall for an ideal $I\subseteq A$
  and any $A$-module $M$ the definition of \emph{$I$-torsion}
  \[M\longmapsto \colim_{n\ge 0}(\Hom_A(A/I^n,M))\subseteq M\]
 and \emph{$I$-adic completion}   \[M\longmapsto\lim_{n\ge 0} (M\otimes_A A/I^n).\]

\begin{exm}[{\cite{DG2002}}]
  Fix a finitely generated ideal $I\subseteq A$ and let $\C_0$ denote the category of
  perfect complexes having $I$-torsion cohomology. Then $\T_0$ equals
  the category of all complexes in $\T$ having $I$-torsion
  cohomology. The functor $\Ga$ equals the local cohomology functor
  (i.e.\ the right derived functor of $I$-torsion), while $\La$ equals
  the derived completion functor (i.e.\ the left derived functor of
  $I$-adic completion).  Moreover, $\widehat\C$ is triangle equivalent
  to $\bfD^b(\proj \widehat A)$.
\end{exm}
\begin{proof}
  Let $K$ denote the \emph{Koszul complex} given by a finite sequence of
  generators $x_1,\ldots,x_n$ of $I$. This is the object $K=K_n$ in
  $\bfD^b(\proj A)$ that is obtained by setting $K_0=A$ and
  $K_{r}=\Cone(K_{r-1}\xto{x_{r}}K_{r-1})$ for $r=1,\ldots,n$. We view
  $K$ as a dg left module over the dg endomorphism ring $E=\END_A(K)$
  and we view $K^\vee=\HOM_A(K,A)$ as a dg right module over $E$. Let
  $\bfD(E)$ denote the derived category of the category of dg right
  $E$-modules. Then we obtain the following diagram
\begin{equation*}
\begin{tikzcd}[column sep = huge]
  {\T}=\bfD(A)    \arrow{r}[description]{q} &\bfD(E)
    \arrow[swap,yshift=1.5ex]{l}{q_\lambda}\arrow[yshift=-1.5ex]{l}{q_\rho}
\end{tikzcd}
\end{equation*}
where
\[q=\HOM_A(K,-)=-\otimes_A K^\vee\qquad q_\lambda=-\otimes_E K\qquad
  q_\rho=\HOM_E(K^\vee,-).\] In \cite[\S6]{DG2002} it is shown that
$q_\lambda$ identifies $\bfD(E)$ with the category of all complexes in
$\bfD(A)$ having $I$-torsion cohomology, while $q_\rho$ identifies
$\bfD(E)$ with the category of all complexes in $\bfD(A)$ that are
$I$-complete. Moreover, it is shown that $\Ga = q_\lambda\circ q$
computes local cohomology, while $\La = q_\rho\circ q$ yields  derived
completion. Next we compute the graded endomorphisms of $q(A)=K^\vee$ and have
by adjunction \[\Hom_E^*(K^\vee,K^\vee)\cong\Hom^*_A(A,\La A)\cong
  H^*(\La A)\cong \widehat A.\] It follows that
\[\widehat\C=\thick (K^\vee)\simeq\bfD^b(\proj \widehat A).\qedhere\]
\end{proof}

We have a more specific description of $\widehat\C$ when $A$ is local;
it uses the tensor triangulated structure of the derived category of a
commutative ring.

\begin{exm}[{\cite[Theorem~4.1]{BIKP2023}}]\label{ex:BIKP}
  Let $A$ be a commutative noetherian local ring and $\fm$ its maximal
  ideal.  Then the derived category $\T=\bfD(\Mod A)$ is a tensor
  triangulated category. Let $\C_0$ denote the category of perfect
  complexes having $\fm$-torsion cohomology, which is a thick tensor
  ideal of $\T^c$. Then $\widehat\C$ equals the subcategory of
  dualisable (or rigid) objects in the tensor triangulated category
  $\T_0$.
\end{exm}

Let us provide a criterion for the  functor $\C_0\to\widehat\C$ to be
a partial completion of triangulated categories; it covers the previous
example of a local ring with $\C_0$ the category of perfect complexes
having torsion cohomology. Our motivation is the following. If
$\C_0\to\widehat\C$ is a partial completion, then the functor
$\C\to\widehat\C$ taking $X$ to $\Ga X$ induces a functor
\[\Coh\C_0\supseteq\{\Hom(-,X)|_{\C_0}\mid X\in\C\}\xto{\ \g\ }\widehat\C\]
such that
\begin{enumerate}
\item $\g$ is fully faithful  and almost an equivalence, up to the fact that $\Ga\C\subseteq
\Ga\T$ need not be a thick subcategory, and
\item the category $\{\Hom(-,X)|_{\C_0}\mid X\in\C\}$ is explicitly
  given by $\C_0\subseteq\C$.
\end{enumerate}

\begin{prop}
  Let $R$ be a commutative ring and suppose that $\C$ is
  $R$-linear. Suppose also that $\Hom(X,Y)$ is a finitely generated
  $R$-module for all objects $X,Y$ in $\C$ and that $\End(X)$ has
  finite length for each $X$ in $\C_0$.  Then the inclusion
  $\C_0\to\widehat\C$ is a partial completion.
\end{prop}
\begin{proof}
  The assumption implies that for each $X\in\widehat\C$ and $C\in\C_0$
  the $R$-module $\Hom(C,X)$ has finite length. Thus we can apply
  Corollary~\ref{co:artinian}.
\end{proof}

We continue with examples.  Let $A$ be a right coherent ring and
denote by $\Inj A$ the category of all injective $A$-modules. Then the
category of complexes $\bfK(\Inj A)$ (with morphisms the chain
morphisms up to homotopy) is compactly generated, and taking a module
to its injective resolution identifies $\bfD^b(\mod A)$ with the full
subcategory of compact objects; see \cite[Proposition~9.3.12]{Kr2022}
for the noetherian case and \cite[Theorem~4.9]{Kr2015} for the general case.

The following example should be compared with Example~\ref{ex:coherent}.

\begin{exm}
  Let $A$ be a right coherent ring. We consider $\T=\bfK(\Inj A)$ and
  choose $\C_0=\bfD^b(\proj A)$. Then it is easily checked that
  $\C\iso\widehat\C$.
\end{exm}

The next example is more challenging.

\begin{exm}[BG-conjecture]
  Let $k$ be a field and $G$ a finite group. We consider the group
  algebra $kG$ and set $\T=\bfK(\Inj kG)$.  Let $\bfi k$ denote an
  injective resolution of the trivial representation. Then the
  assignment $X\mapsto X\otimes_k \bfi k$ identifies
  $\C=\bfD^b(\mod kG)$ with $\T^c$. We choose $\C_0=\thick(k)$ and
  note that $\T_0$ identifies with the derived category of dg modules
  over the algebra $C^*(BG;k)$ (the cochains of the classifying space
  $BG$ with coefficients in $k$) \cite{BK2008}. For $X\in\T$ we
  set
  \[H^*(G,X)=\Hom^*(\bfi k,X)=\bigoplus_{n\in\bbZ}\Hom(\bfi k,\Si^n
    X)\] and view this as a module over the cohomology algebra
  $H^*(G,k)$. It is well known that this algebra is noetherian and
  that $H^*(G,X)$ is a noetherian module for any compact object $X$. A
  recent conjecture of Benson and Greenlees asserts that $\widehat\C$
  equals the category of objects in $\T_0$ such that $H^*(G,X)$ is a
  noetherian module over $H^*(G,k)$; see
  \cite[Conjecture~1.4]{BG2023}. Note that $\C_0\to\widehat\C$ is a
  partial completion.
\end{exm}

We offer another challenge.

\begin{exm}[Strong generation conjecture]
  Let $A$ be a commutative noetherian local ring. It is well known
  that $A$ is regular if and only if the triangulated category
  $\bfD^b(\proj A)$ admits a strong generator (in the sense of Bondal
  and Van den Bergh \cite{BV2003}). On the other hand, $A$ is regular
  if and only if its completion $\widehat A$ is regular. Keeping in
  mind Example~\ref{ex:BIKP}, one may conjecture: If $\C$ admits a strong
  generator, then $\widehat\C$ admit a strong generator.
\end{exm}

\section{Completing complexes of finite length modules}

We return to a  previous example and give now a full proof of the
following.

\begin{prop}[{\cite[Example~4.2]{Kr2020}}]\label{pr:flA}
  For a commutative noetherian ring $A$ the inclusion
  \[\bfD^b(\fl A)\lto \bfD^b(\art A)\] is a sequential partial completion.
\end{prop}

\begin{proof}
  Recall that an $A$-module $M$ is artinian if and only if $M$ is the
  union of finite length modules and $\soc M$
  has finite length \cite[Proposition~2.4.20]{Kr2022}. In particular,
  the abelian category $\art A$ has enough injective objects.

  We write $\Mod_0 A$ for the full subcategory of $A$-modules that are
  filtered colimits of finite length modules.  Thus the inclusion
  $\fl A\to\Mod A$ induces an equivalence $\Ind\fl A\iso\Mod_0 A$.
  Set
  \[\Inj_0 A= \Inj A \cap\Mod_0 A\qquad\text{and}\qquad \inj_0
    A=\Inj A\cap\art A.\] Note that an injective $A$-module is
  in $\Mod_0 A$ if and only if each indecomposable direct summand is
  artinian.  Now consider the compactly generated triangulated
  $\T=\bfK(\Inj_0 A)$ given by the complexes in $\Inj_0 A$. Then we
  have canonical triangle equivalences
  \[\bfD^b(\fl A)\iso\T^c \qquad\text{and}\qquad \bfD^b(\art
    A)\iso\bfK^{+,b}(\inj_0 A)\subseteq \T;\] see Corollary~4.2.9 and
  Proposition~9.3.12 in \cite{Kr2022}. Next observe that for
  $X\in \bfD^b(\fl A)$ and $Y \in\bfD^b(\art A)$ the $A$-module
  $\Hom(X,Y)$ has finite length. This amounts to showing that
  $\Ext^i(M,N)$ has finite length for all $M\in\fl A$, $N\in\art A$,
  and $i\in\bbZ$, which reduces to the case $i=0$ by taking an
  injective resolution of $N$.  Thus we can apply
  Corollary~\ref{co:artinian} and it follows that
  $\bfD^b(\fl A)\to \bfD^b(\art A)$ is a partial completion. It remains to
  observe that each complex $X$ in $\bfD^b(\art A)$ is  the colimit of the
  sequence $(\soc ^n X)_{n\ge 0}$ in $\bfD^b(\fl A)$, but this
  need not be a Cauchy sequence.
\end{proof}

When the ring $A$ is local and regular, the completion  $\bfD^b(\fl
A)\to \bfD^b(\art A)$  identifies with $\C_0\to\widehat\C$ in Example~\ref{ex:BIKP}.
  
\section{Completing perfect complexes}

We return to another example and give  a full proof of the
following.

\begin{prop}[{\cite[Theorem~6.2]{Kr2020}}]\label{pr:perfect}
  For a right coherent ring $A$  the inclusion
  \[\bfD^b(\proj A)\lto\bfD^b(\mod A)\] is a Cauchy sequential partial
  completion.
\end{prop}

The proof requires several lemmas which are of independent
interest. Let $\T$ be a triangulated category with suspension
$\Si\colon\T\iso\T$ and suppose that countable coproducts exist in
$\T$.

\begin{defn}
 A \emph{homotopy
colimit}  of a sequence of morphisms 
\[
\begin{tikzcd}
X_0\arrow{r}{\p_0}&X_1\arrow{r}{\p_1}&
  X_2\arrow{r}{\p_2}&\cdots
\end{tikzcd}\] in $\T$ is an object $X$ that occurs
in an exact triangle
\[\begin{tikzcd}
    \Si^{-1}X\arrow{r}&\coprod_{n\ge
      0}X_n\arrow{r}{\id-\p}&\coprod_{n\ge 0}X_n\arrow{r}& X.
  \end{tikzcd}\] We write $\hocolim_n X_n$ for $X$ and observe that a
homotopy colimit is unique up to a (non-unique) isomorphism.
\end{defn}

Recall that an object $C$ in $\T$ is \emph{compact} if $\Hom(C,-)$
preserves all coproducts. A morphism $X\to Y$ is \emph{phantom} if any
composition $C\to X\to Y$ with $C$ compact is zero. The phantom
morphisms form an ideal and we write $\Ph(X,Y)$ for the subgroup of
all phantoms in $\Hom(X,Y)$.

Let us compute the functor $\Hom(-, \hocolim_n X_n)$. To this end observe that a
sequence
\[\begin{tikzcd}
A_0\arrow{r}{\p_0}&A_1\arrow{r}{\p_1}&
  A_2\arrow{r}{\p_2}&\cdots
\end{tikzcd}\]
of maps between abelian groups induces an exact sequence
\[\begin{tikzcd}
    0\arrow{r}&\coprod_{n\ge
      0}A_n\arrow{r}{\id-\p}&\coprod_{n\ge 0}A_n\arrow{r}&   \colim_n A_n\arrow{r}&0
\end{tikzcd}\]
because it identifies with the colimit of the exact sequences
\[\begin{tikzcd}
    0\arrow{r}&\coprod_{i=0}^{n-1}A_i\arrow{r}{\id-\p}&\coprod_{i=0}^{n}A_i\arrow{r}&  A_n\arrow{r}&0.
\end{tikzcd}\]

\begin{lem}\label{le:hocolim}
  Let $C\in\T$ be compact. Then any sequence $X_0\to X_1\to X_2\to\cdots$ in
  $\T$ induces an isomorphism
  \[\colim_n\Hom(C,X_n)\longiso\Hom(C,\hocolim_n X_n).\]
\end{lem}
\begin{proof}
  The above observation gives an exact sequence
  \[\begin{tikzcd}[column sep=small]
      0\arrow{r}&
      \coprod_{n}\Hom(C,X_n)\arrow{r}&\coprod_{n}\Hom(C,X_n)\arrow{r}&
      \colim_n\Hom(C,X_n)\arrow{r}&0.
  \end{tikzcd}\] 
Now apply $\Hom(C,-)$ to the defining
  triangle for $\hocolim_n X_n$. Comparing both sequences yields the
  assertion, since \[\coprod_n\Hom(C,X_n)\cong \Hom(C, \coprod_n X_n).\qedhere\]
\end{proof}

Recall that for any sequence
$\cdots\to A_2\xto{\p_2} A_1\xto{\p_1} A_0$ of maps between abelian
groups the inverse limit and its first derived functor are given by
the exact sequence
\[\begin{tikzcd}
  0\lto\lim_nA_n\lto  \prod_{n\ge 0}A_n\arrow{r}{\id-\p}&\prod_{n\ge 0}A_n\arrow{r}&
    \lim^1_nA_n\arrow{r}& 0.
  \end{tikzcd}\]
Note that  $\lim^1_nA_n=0$ when $A_n\iso A_{n+1}$ for $n\gg 0$.

The following result goes back to work of Milnor \cite{Mi1962} and 
has been extended by several authors.

\begin{lem}\label{le:phantom}
  Let $X=\hocolim_n X_n$ be a homotopy colimit in $\T$ such that
  each $X_n$ is a coproduct of compact objects. Then we have for any
  $Y$ in $\T$ a natural exact sequence
\[0\lto\Ph(X,Y)\lto\Hom(X,Y)\lto\lim_n\Hom(X_n,Y)\lto 0\]
and an isomorphism
\[\Ph(X,\Si Y)\cong \sideset{}{^{1}}\lim_n\Hom(X_n,Y).\]
\end{lem}
\begin{proof}
  Apply $\Hom(-,Y)$ to the exact triangle defining $\hocolim_n X_n$
  and use that a morphism $X\to Y$ is phantom if and only if it
  factors through the canonical morphism
  $X\to\coprod_{n\ge 0}\Si X_n$.
\end{proof}

Let $\C\subseteq\T$ be a full triangulated subcategory consisting of
compact objects and consider the restricted Yoneda functor
\[\T\lto\Coh\C,\quad X\mapsto h_X:=\Hom(-,X)|_\C.\]
This functor induces for each pair of objects $X,Y\in\T$ a map
\[\Hom(X,Y)\lto\Hom(h_X,h_Y).\]
Clearly, this map is bijective when $X$ is in $\C$, and it remains bijective
when $X$ is a coproduct of objects in $\C$. 

\begin{lem}\label{le:yoneda}
  Let $X=\hocolim_n X_n$ be a homotopy colimit in $\T$ such that
  each $X_n$ is a coproduct of objects in $\C$. Then we have for any
  $Y$ in $\T$ a natural isomorphism
  \[\Hom(X,Y)/\Ph(X,Y)\iso\Hom(h_X,h_Y).\]
  \end{lem}
  \begin{proof}
   We have
    \begin{align*}
      \Hom(X,Y)/\Ph(X,Y)&\cong\lim_n\Hom(X_n,Y)\\
                                &\cong\lim_n\Hom(h_{X_n},h_{Y})\\
                                &\cong\Hom(\colim_nh_{X_n},h_{Y})\\
      &\cong\Hom(h_X,h_Y).
    \end{align*}
    The first isomorphism follows from Lemma~\ref{le:phantom}, the
    second uses that each $X_n$ is a coproduct of objects in $\C$, the
    third is clear, and the last follows from
    Lemma~\ref{le:hocolim}.
 \end{proof}   

\begin{proof}[Proof of Proposition~\ref{pr:perfect}]
  Set $\P=\proj A$. The inclusion $\proj A\to\mod A$ induces a triangle equivalence
\[\bfK^{-,b}(\proj A)\longiso\bfD^b(\mod A).\] Thus we may identify $\bfD^b(\proj A)\to\bfD^b(\mod A)$
with the inclusion
\[\bfK^b(\P)\lto \bfK^{-,b}(\P).\]
We think of these as subcategories of $\bfK(\Proj A)$, where $\Proj A$
denotes the category of all projective $A$-modules. In particular
$\bfK(\Proj A)$ has arbitrary coproducts and all objects from
$\bfK^b(\P)$ are compact.
  
For any complex $X$ we consider the sequence of truncations
\[\cdots\lto \s_{\ge n+1}X\lto \s_{\ge n}X\lto \s_{\ge n-1}X\lto\cdots\]
given by
\[
\begin{tikzcd}
\s_{\ge n}X\arrow{d}& \cdots \arrow{r}&
  0\arrow{r}{}\arrow{d}&0\arrow{r}{}\arrow{d}&
X^{n}\arrow{r}\arrow{d}{\id}&X^{n+1}\arrow{r}\arrow{d}{\id}&\cdots \\ 
X&\cdots\arrow{r}&
X^{n-2}\arrow{r}{}&X^{n-1}\arrow{r}{}&X^n\arrow{r}&X^{n+1}\arrow{r}&\cdots .
\end{tikzcd}
\]
For $X$  in $\bfK^{-,b}(\P)$ and $n\in\bbZ$  we set
$X_n= \s_{\ge -n}X$. This yields a Cauchy sequence
\[ X_0\lto X_1\lto X_2\lto\cdots\] in $\bfK^b(\P)$ with
$\hocolim_{n\ge 0}X_n\cong X$.

We claim that the restricted Yoneda functor
  \[\bfK^{-,b}(\P)\lto \Coh\bfK^b(\P),\quad
    X\mapsto h_X:=\Hom(-,X)|_{\bfK^b(\P)},\] is fully faithful.  Let
  $X,Y$ be objects in $\bfK^{-,b}(\P)$. As before we write $X$ as homotopy
  colimit of its truncations $X_n=\s_{\ge -n}X$ and denote by $C_n$  the
  cone of $X_{n-1}\to X_{n}$. This complex is concentrated in degree
  $-n$; so $\Hom(C_n,Y)=0$ for $n\gg 0$. Thus $X_{n}\to X_{n+1}$
  induces a bijection
\[\Hom(X_{n+1},Y)\longiso\Hom(X_{n},Y) \quad\text{for}\quad n\gg
  0.\] This implies
 \[\Hom(X,Y)\longiso\lim_n\Hom(X_n,Y)\]
 and therefore $\Ph(X,Y)=0$ by Lemma~\ref{le:phantom}.  From
 Lemma~\ref{le:yoneda} we conclude that
  \[\Hom(X,Y)\longiso\Hom(h_X,h_Y).\qedhere\]
\end{proof}

From the  proof of Proposition~\ref{pr:perfect}  we learn that  each
complex in $\bfD^b(\mod A)$ is not only a filtered colimit of perfect
complexes; it is actually the colimit of a Cauchy sequence which is
obtained from its truncations. In particular, we are in the situation
that a homotopy colimit is an honest colimit, and therefore unique up to a
unique isomorphism.

\section{Completion using enhancements}

While the ind-completion of a category is a fairly explicit
construction, it is not immediately clear how to deal with additional
structure. In particular, there is no obvious triangulated structure
for $\Ind\C$ when $\C$ is triangulated. One way to address this
problem is the use of enhancements.

Recall that a triangulated category is \emph{algebraic} if it is
triangle equivalent to the stable category $\St\A$ of a Frobenius
category $\A$. 
A morphism between exact triangles
\[
\begin{tikzcd}
  X\arrow{r}\arrow{d}&Y\arrow{r}\arrow{d}&
  Z\arrow{r}\arrow{d}&\Si
  X\arrow{d}\\
  X'\arrow{r}&Y'\arrow{r}&Z'\arrow{r}&\Si X'
\end{tikzcd}
\]
in $\St\A$ will be called \emph{coherent} if it can be lifted to
a morphism
\[
\begin{tikzcd}
  0\arrow{r}&\tilde X\arrow{r}\arrow{d}& \tilde
  Y\arrow{r}\arrow{d}&
  \tilde  Z\arrow{r}\arrow{d}&0\\
  0\arrow{r}&\tilde X'\arrow{r}&\tilde
  Y'\arrow{r}&\tilde Z'\arrow{r}&0
\end{tikzcd}
\]
between  exact sequences in $\A$ so that the canonical functor
$\A\to\St\A$ maps the second to the first diagram.

\begin{defn}
  Let $\C$ be a triangulated category and $\X$ a class of sequences in
$\C$.  We
say that $\X$ is \emph{phantomless} if for any pair of sequences $X,Y$
in $\X$ we have
\begin{equation*}\label{eq:ph-less}
  \sideset{}{^{1}}\lim_i\colim_j\Hom(X_i,Y_j)=0.
\end{equation*}
\end{defn}

This definition is consistent with our previous discussion of phantom
morphisms in the following sense. Let $\C\subseteq\T$ be a
triangulated subcategory consisting of compact objects such that $\T$
admits countable coproducts. Suppose that $\X$ is \emph{stable under
  suspensions}, i.e.\ $(X_i)_{i\ge 0}$ in $\X$ implies
$(\Si^nX_i)_{i\ge 0}$ in $\X$ for all $n\in\bbZ$. Then $\X$ is
phantomless if and only if \[\Ph(\hocolim_i X_i, \hocolim_j Y_j)=0\]
for all $X,Y$ in $\X$. This follows from Lemmas~\ref{le:hocolim} and
\ref{le:phantom}.

\begin{prop}[{\cite[Theorem~4.7]{Kr2020}}]\label{pr:tria-completion}
  Let $\C$ be an algebraic triangulated category, viewed as a full
  subcategory of its sequential completion $\Ind_\bbN\C$. Let $\X$ be
  a class of sequences in $\C$ that is phantomless, closed
  under suspensions, and closed under the formation of cones. Then the full
  subcategory $\Ind_\X\C\subseteq\Ind_\bbN\C$ given by the colimits of
  sequences in $\X$ admits a unique triangulated structure such
  that the exact triangles are precisely the ones isomorphic to
  colimits of sequences that are given by coherent morphisms of
  exact triangles in $\C$.
\end{prop}

Let us spell out the triangulated structure for $\Ind_\X\C$. Fix a
sequence of coherent morphisms $\eta_0\to\eta_1\to\eta_2\to\cdots$ of
exact triangles
  \[\eta_i\colon X_i\lto Y_i\lto Z_i\lto\Si X_i\] in $\C$ and suppose
  that it is also a sequence of morphisms $X\to Y\to Z\to \Si X$ of 
  sequences in $\X$. This identifies with the sequence
  \[\colim_i X_i \lto \colim_i Y_i\lto \colim_i Z_i\lto \colim_i\Si X_i\]
in $\Ind_\X\C$, and the exact triangles in  $\Ind_\X\C$ are
precisely sequences of morphisms that are isomorphic to sequences of
the above form.

\begin{proof}[Proof of Proposition~\ref{pr:tria-completion}]
  We use the enhancement as follows. Suppose that $\C=\St\A$ for some
  Frobenius category $\A$. We denote by $\C\tilde{\phantom{e}}$ the
  stable category $\St\A\tilde{\phantom{e}}$ of the countable envelope
  of $\A$; see Example~\ref{ex:countable-envelope}. This is a
  triangulated category with countable coproducts and $\C$ identifies
  with a full subcategory of compact objects.

  Given sequences $X,Y$ in
  $\X$ we set $\bar X=\hocolim_i X_i$ and $\bar Y=\hocolim_j Y_j$ in
  $\C\tilde{\phantom{e}}$. Using that $\X$ is phantomless we compute
\[\lim_i\colim_j\Hom(X_i,Y_j)\cong\Hom(h_X,h_Y)\cong\Hom(\bar X,\bar
  Y)\] where the first isomorphism follows from
Lemma~\ref{le:hocolim} and the second from Lemma~\ref{le:yoneda}.
Thus taking a sequence in $\X$ to its homotopy colimit in
$\C\tilde{\phantom{e}}$ provides a fully faithful functor
\[\hocolim\colon\Ind_\X\C\lto\C\tilde{\phantom{e}}.\]
Then it remains to compare the triangulated structures on both side,
which turn out to be equivalent by construction.
\end{proof}

The above result admits a substantial generalisation, from algebraic
triangulated categories to triangulated categories with a \emph{morphic
enhancement} in the sense of Keller
\cite[Appendix~C]{Kr2020}. Moreover, in some interesting cases the
morphic enhancement extends to a morphic enhancement of the
completion.

\begin{exm}
  For a right coherent ring $A$ let $\X$ denote the class of Cauchy
  sequences $(X_i)_{i\ge 0}$ in $\bfD^b(\proj A)$ such that
  $\colim_i H^n(X_i)=0$ for $|n|\gg 0$. Then $\X$ is phantomless and
  we have a triangle equivalence
  \[\Ind_\X\bfD^b(\proj A)\iso\bfD^b(\mod A).\]
\end{exm}

We end these notes with a couple of references that complement
our approach towards the completion of triangulated categories. The work
of Neeman offers an intriguing approach that uses metrics on
triangulated categories, thereby avoiding the use of any enhancements
\cite{Ne2018,Ne2020}. On the other hand, there is Lurie's approach via
stable $\infty$-categories \cite[\S1]{Lu2017}; it uses a notion of 
enhancement that is far more sophisticated than the one presented in
these notes.

\end{document}